\documentclass[11pt]{amsart}
\usepackage{amsfonts,amssymb,amsthm}
\usepackage{amsmath,amscd}
\usepackage{pstricks}
\usepackage{pstricks,pst-node}
\usepackage{mathrsfs}
\usepackage[all]{xy}

\DeclareMathAlphabet{\mathpzc}{OT1}{pzc}{m}{it}

\renewcommand{\subsection}[1]{\vspace{.18in}
\par\noindent\addtocounter{subsection}{1}
\setcounter{equation}{0}{\bf\thesubsection.\hspace{5pt}#1}}

\theoremstyle{definition}
\newtheorem{Def}[subsection]{Definition}

\theoremstyle{plain}
\newtheorem{Prop}[subsection]{Proposition}
\newtheorem{Thm}[subsection]{Theorem}

\newtheorem{Lem}[subsection]{Lemma}
\newtheorem{Coro}[subsection]{Corollary}
\numberwithin{equation}{subsection}

\def\bffkk{\boldsymbol{\frak k}}

\newcommand{\bfa}{{\mathbf{a}}}

\newcommand{\bfi}{{\mathbf{i}}}

\newcommand{\bfs}{{\mathbf{s}}}

\newcommand{\bfP}{{\mathbf{P}}}
\newcommand{\bfQ}{{\mathbf{Q}}}

\def\fS{{\frak S}}

\newcommand{\ms}{\mathscr}

\newcommand{\msD}{\mathscr D}

\def\sfz{{\mathsf z}}
\def\sfs{{\mathsf s}}

\newcommand{\mc}{\mathcal}

\def\sF{{\mathsf F}}
\def\sG{{\mathsf G}}
\def\sH{{\mathcal H}}
\def\sI{{\mathcal I}}
\def\sJ{{\mathcal J}}

\def\sP{{\mathcal P}}
\def\sQ{{\mathcal Q}}

\def\sS{{\mathcal S}}
\def\sT{{\mathcal T}}

\def\sX{{\mathcal X}}
\def\sY{{\mathcal Y}}

\newcommand{\mbnn}{\mathbb N^{n}}
\newcommand{\mbn}{\mathbb N}

\newcommand{\mbc}{\mathbb C}
\newcommand{\mbz}{\mathbb Z}

\newcommand{\ttk}{\mathtt{k}}
\newcommand{\tth}{\mathtt{h}}

\newcommand{\ttx}{\mathtt{x}}
\newcommand{\ttg}{\mathtt{g}}

\newcommand{\End}{\operatorname{End}}

\newcommand{\spann}{\operatorname{span}}

\newcommand{\la}{{\lambda}}
\newcommand{\La}{\Lambda}

\newcommand{\dt}{\delta}
\newcommand{\Dt}{\Delta}

\newcommand{\Og}{\Omega}
\newcommand{\og}{\omega}

\newcommand{\vi}{\varphi}
\newcommand{\ep}{\varepsilon}
\newcommand{\al}{\alpha}

\newcommand{\sg}{\sigma}

\def\th{\theta} 

\newcommand{\ol}{\overline}

\newcommand{\Lcp}{\bar L}
\newcommand{\Mcp}{\bar M}
\newcommand{\Icp}{\bar I}

\newcommand{\leb}{\left[}

\newcommand{\rib}{\right]}

\def\ggp#1#2{\left[\kern-3.2pt\left[{#1\atop #2}\right]\kern-3.2pt\right]}

\def\hmod{{\text-}{\mathsf{mod}}}

\def\leq{\leqslant}\def\geq{\geqslant}
\def\le{\leqslant}

\newcommand{\op}{\oplus}
\newcommand{\ot}{\otimes}

\newcommand{\han}{\subseteq}

\newcommand{\h}{\widehat}
\newcommand{\ti}{\widetilde}

\newcommand{\Lanr}{\Lambda(n,r)}

\newcommand{\lra}{\longrightarrow}
\newcommand{\ra}{\rightarrow}

\newcommand{\lm}{\longmapsto}
\newcommand{\pa}{\partial}

\newcommand{\zrC}{{\zeta}_{r}}

\newcommand{\vtg}{{\!\vartriangle\!}}

\newcommand{\DC}{{\frak D}_{\vtg,\mathbb C}}

\def\ttv{v}

\newcommand{\SrC}{\sS(n,r)_\mbc}
\newcommand{\UglC}{\text{\rm U}_{\mathbb C}({\frak{gl}}_n)}
\newcommand{\afHrC}{\sH_{\vtg}(r)_\mathbb C}
\newcommand{\afUslC}{\text{\rm U}_{\mathbb C}(\widehat{\frak{sl}}_n)}

\newcommand{\HrC}{\sH(r)_{\mathbb C}}

\newcommand{\afSrC}{{\mathcal S}_{\vtg}(n,r)_{\mathbb{C}}}

\newcommand{\afUglC}{\text{\rm U}_{\mathbb C}(\widehat{\frak{gl}}_n)}

\newcommand{\fSr}{\fS_r}

\newcommand{\OgC}{\Og_{\mathbb C}}
\newcommand{\OgnC}{\Og_{n,\mathbb C}}

\newcommand{\tri}{\triangle(n)}

\begin{document}
\title{Affine quantum Schur algebras and affine Hecke algebras}
\author{Qiang Fu}
\address{Department of Mathematics, Tongji University, Shanghai, 200092, China.}
\email{q.fu@hotmail.com}


\thanks{Supported by the National Natural Science Foundation
of China, the Program NCET, Fok Ying Tung Education Foundation
and the Fundamental Research Funds for the Central Universities}

\begin{abstract}
Let $\sF$ be the Schur functor from the category of finite dimensional $\afHrC$-modules to the category of finite dimensional $\afSrC$-modules, where $\afHrC$ is the extended affine Hecke algebra of type $A$ over $\mbc$ and $\afSrC$ is the affine quantum Schur algebras over $\mbc$.
The Drinfeld polynomials associated with $\sF(V)$ were determined in  \cite[7.6]{CP96} and \cite[4.4.2]{DDF} in the case of $n>r$, where $V$ is an irreducible  $\afHrC$-module. We will generalize the result in [loc. cit.] to the case of $n\leq r$. As an application, we will classify finite dimensional irreducible $\afSrC$-modules, which has been proved in \cite[4.6.8]{DDF} using a different method. Furthermore we will use it to generalize \cite[(6.5f)]{Gr80} to the affine case.
\end{abstract}
 \sloppy \maketitle
\section{Introduction}
It is well known that finite dimensional irreducible modules for quantum affine algebras were classified by Chari--Pressley in terms of Drinfeld polynomials (cf. \cite{CP91,CPbk,CP95,CP97}). Finite dimensional irreducible modules for $\afHrC$ were classified in \cite{Zelevinsky,Rogawski}, where $\afHrC$ is the extended affine Hecke algebra of type $A$ over the complex field $\mbc$ with a non-root of unity.
The category of finite dimensional $\afHrC$-modules and the category of finite dimensional $\afUslC$-modules which are of level $r$ are related by a functor $\mc F$, which was defined in \cite[4.2]{CP96}. Here $\afUslC$ is quantum affine $\frak{sl}_n$ over $\mbc$.
Chari--Pressley  proved in [loc. cit.] $\mc F$ is an equivalence of categories if $n>r$. Furthermore the Drinfeld polynomials associated with $\mc F(V)$ were determined in [loc. cit. 7.6] in the case of $n>r$, where $V$ is an irreducible $\afHrC$-module.

Let $\afUglC$ be quantum affine $\frak{gl}_n$ over $\mbc$.
In \cite{FM}, finite dimensional irreducible polynomial representations of $\afUglC$ were classified. It was proved in \cite[3.8.1]{DDF} that the natural algebra homomorphism $\zeta_r$ from $\afUglC$ to the affine quantum Schur algebra $\afSrC$ is surjective.
Every  $\afSrC$-module can be regarded as a $\afUglC$-module via $\zrC$. Let $\sF$ be the Schur functor from the category of finite dimensional $\afHrC$-modules to the category of finite dimensional $\afSrC$-modules. It was proved in \cite[4.1.3 and 4.2.1]{DDF} that $\sF$ is an equivalence of categories in the case of $n\geq r$ and $\sF(V)|_{\afUslC}$ is isomorphic to $\mc F(V)$ for any $\afHrC$-module $V$.
Furthermore, using \cite[7.6]{CP96}, the Drinfeld polynomials associated with $\sF(V)$ were determined in \cite[4.4.2]{DDF} in the case of $n>r$, where $V$ is an irreducible $\afHrC$-module. We will generalize \cite[7.6]{CP96} and \cite[4.4.2]{DDF} to the case of $n\leq r$ in \ref{main theorem}. Using this result, we will prove in \ref{classification 2} the classification theorem of finite dimensional irreducible $\afSrC$-modules, which was established in \cite[4.6.8]{DDF}. Finally, we will
relate the parametrization of irreducible $\sS_\vtg(N,r)_\mbc$-modules, via the functor $\sG$ defined in \eqref{functor sG}, to the parametrization of irreducible $\afSrC$-modules in \ref{prop of functor sG}. This result is the affine version of \cite[(6.5f)]{Gr80}.

\section{Quantum affine $\frak{gl}_n$}

Let $\ttv\in\mbc^*$ be a complex number which is not a root of unity,
where $\mbc^*=\mbc\backslash\{0\}$. Let $(c_{i,j})$ be the Cartan matrix of affine type $A_{n-1}$. We recall the Drinfeld's new realization of quantum affine $\frak{gl}_n$ as follows.

\begin{Def}\label{QLA}
 The {\it quantum loop algebra} $\afUglC$  (or {\it
quantum affine $\mathfrak {gl}_n$})  is the $\mbc$-algebra generated by $\ttx^\pm_{i,s}$
($1\leq i<n$, $s\in\mbz$), $\ttk_i^{\pm1}$ and $\ttg_{i,t}$ ($1\leq
i\leq n$, $t\in\mbz\backslash\{0\}$) with the following relations:
\begin{itemize}
 \item[(QLA1)] $\ttk_i\ttk_i^{-1}=1=\ttk_i^{-1}\ttk_i,\,\;[\ttk_i,\ttk_j]=0$,
 \item[(QLA2)]
 $\ttk_i\ttx^\pm_{j,s}=\ttv^{\pm(\dt_{i,j}-\dt_{i,j+1})}\ttx^\pm_{j,s}\ttk_i,\;
               [\ttk_i,\ttg_{j,s}]=0$,
 \item[(QLA3)] $[\ttg_{i,s},\ttx^\pm_{j,t}]
               =\begin{cases}0,\;\;&\text{if $i\not=j,\,j+1$};\\
                  \pm \ttv^{-js}\frac{[s]}{s}\ttx^\pm_{j,s+t},\;\;\;&\text{if $i=j$};\\
                  \mp \ttv^{-js}\frac{[s]}{s}\ttx_{j,s+t}^\pm,\;\;\;&\text{if $i=j+1$,}
                \end{cases}$
 \item[(QLA4)] $[\ttg_{i,s},\ttg_{j,t}]=0$,
 \item[(QLA5)]
 $[\ttx_{i,s}^+,\ttx_{j,t}^-]=\dt_{i,j}\frac{\phi^+_{i,s+t}
 -\phi^-_{i,s+t}}{\ttv-\ttv^{-1}}$,
 \item[(QLA6)] $\ttx^\pm_{i,s}\ttx^\pm_{j,t}=\ttx^\pm_{j,t}\ttx^\pm_{i,s}$, for $|i-j|>1$, and
 $[\ttx_{i,s+1}^\pm,\ttx^\pm_{j,t}]_{\ttv^{\pm c_{ij}}}
               =-[\ttx_{j,t+1}^\pm,\ttx^\pm_{i,s}]_{\ttv^{\pm c_{ij}}}$,
 \item[(QLA7)]
 $[\ttx_{i,s}^\pm,[\ttx^\pm_{j,t},\ttx^\pm_{i,p}]_\ttv]_\ttv
 =-[\ttx_{i,p}^\pm,[\ttx^\pm_{j,t},\ttx^\pm_{i,s}]_\ttv]_\ttv\;$ for
 $|i-j|=1$,
\end{itemize}
 where $[x,y]_a=xy-ayx$, $[s]=\frac{v^s-v^{-s}}{v-v^{-1}}$  and $\phi_{i,s}^\pm$ are defined via the
 generating functions in indeterminate $u$ by
$$\Phi_i^\pm(u):={\ti\ttk}_i^{\pm 1}
\exp\bigl(\pm(\ttv-\ttv^{-1})\sum_{m\geq 1}\tth_{i,\pm m}u^{\pm
m}\bigr)=\sum_{s\geq 0}\phi_{i,\pm s}^\pm u^{\pm s}$$ with
$\ti\ttk_i=\ttk_i/\ttk_{i+1}$ ($\ttk_{n+1}=\ttk_1$) and $\tth_{i,\pm
m}=\ttv^{\pm(i-1)m}\ttg_{i,\pm m}-\ttv^{\pm(i+1)m}\ttg_{i+1,\pm
m}\,(1\leq i<n).$
\end{Def}

The algebra $\afUglC$ has another presentation which we now describe.
Let $\DC(n)$ be the double Ringel--Hall algebra of the cyclic quiver $\tri$. By \cite[2.3.1]{DDF}, the algebra $\DC(n)$ has the following presentation.

\begin{Lem} \label{presentation dHallAlg} The double Ringel--Hall algebra $\DC(n)$ of
the cyclic quiver $\tri$ is the $\mbc$-algebra generated by
$E_i,\ F_i,\  K_i,\ K_i^{-1},\ \sfz^+_s,\ \sfz^-_s,$ for $1\leq i\leq n,\
s\in\mbz^+$, and relations:
\begin{itemize}
\item[(QGL1)] $K_{i}K_{j}=K_{j}K_{i},\ K_{i}K_{i}^{-1}=1$;

\item[(QGL2)] $K_{i}E_j=\ttv^{\dt_{i,j}-\dt_{i,j+1}}E_jK_{i}$,
$K_{i}F_j=\ttv^{-\dt_{i, j}+\dt_{ i,j+1}} F_jK_i$;

\item[(QGL3)] $E_iF_j-F_jE_i=\delta_{i,j}\frac
{\ti K_{i}-{\ti K_{i}}^{-1}}{\ttv-\ttv^{-1}}$, where $\ti K_i=
K_iK_{i+1}^{-1}$;

\item[(QGL4)]
$\displaystyle\sum_{a+b=1-c_{i,j}}(-1)^a\leb{1-c_{i,j}\atop a}\rib
E_i^{a}E_jE_i^{b}=0$ for $i\not=j$;

\item[(QGL5)]
$\displaystyle\sum_{a+b=1-c_{i,j}}(-1)^a\leb{1-c_{i,j}\atop a}\rib
F_i^{a}F_jF_i^{b}=0$ for $i\not=j$;

\item[(QGL6)] $\sfz^+_s\sfz^+_t=\sfz^+_t\sfz^+_s$,$\sfz^-_s\sfz^-_t=\sfz^-_t\sfz^-_s$,  $\sfz^+_s\sfz^-_t=\sfz^-_t\sfz^+_s$;
 \item[(QGL7)] $K_i\sfz^+_s=\sfz^+_s K_i$, $K_i\sfz^-_s=\sfz^-_s K_i$;

 \item[(QGL8)] $E_i\sfz^+_s=\sfz^+_s E_i$, $E_i\sfz^-_s=\sfz^-_s E_i,$ $F_i\sfz^-_s=\sfz^-_s F_i$, and  $\sfz^+_s F_i=F_i\sfz^+_s$,
\end{itemize}
where $1\leq i,j\leq n$, $s,t\in \mbz^+$ and $\big[{c\atop a}\big]=\prod_{s=1}^a\frac{v^{c-s+1}-v^{-c+s-1}}{v^s-v^{-s}}$ for $c\in\mbz$.
It is a Hopf algebra with
comultiplication $\Dt$, counit $\ep$, and antipode $\sg$ defined
by
\begin{eqnarray*}\label{Hopf}
&\Delta(E_i)=E_i\otimes\ti K_i+1\otimes
E_i,\quad\Delta(F_i)=F_i\otimes
1+\ti K_i^{-1}\otimes F_i,&\\
&\Delta(K^{\pm 1}_i)=K^{\pm 1}_i\otimes K^{\pm 1}_i,\quad
\Delta(\sfz_s^\pm)=\sfz_s^\pm\otimes1+1\otimes
\sfz_s^\pm;&\\
&\varepsilon(E_i)=\varepsilon(F_i)=0=\varepsilon(\sfz_s^\pm),
\quad \varepsilon(K_i)=1;&\\
&\sg(E_i)=-E_i\ti K_i^{-1},\quad \sg(F_i)=-\ti K_iF_i,\quad
\sg(K^{\pm 1}_i)=K^{\mp 1}_i,&\\
&\text{and}\;\;\sg(\sfz_s^\pm)=-\sfz_s^\pm,&
\end{eqnarray*}
 where $1\leq i\leq n$ and $s\in \mbz^+$.
\end{Lem}

Let $\afUslC$ be the subalgebra of $\DC(n)$ generated by $E_i,\ F_i,\  \ti K_i,\ti K_i^{-1}$ for $i\in [1,n]$. Beck \cite{Be} proved that $\afUslC$ is isomorphic to
the subalgebra of $\afUglC$ generated by all $\ttx^\pm_{i,s}$, $\ti\ttk_i^{\pm1}$ and $\tth_{i,t}$.
The following result extends Beck's isomorphism..

\begin{Lem}[{\cite[4.4.1]{DDF}}] \label{DDFIsoThm}
There is a Hopf  algebra isomorphism
 $$f:\DC(n)\lra
\afUglC$$ such that
$$\aligned
&K_i^{\pm1}\lm\ttk_i^{\pm1},\;\quad
E_j\lm \ttx^+_{j,0},\;\quad
F_j\lm \ttx^-_{j,0}
\;(1\leq i\leq n,\,1\leq j<n),\;\\
&E_n\lm \ttv\sX \ti\ttk_n,
\quad F_n\lm\ttv^{-1}\ti\ttk_n^{-1}\sY,
\quad
\sfz^\pm_s\lm
\mp s\ttv^{\pm s}\th_{\pm
s}\;(s\geq 1),\\
\endaligned$$
where
$
\th_{\pm s} =\mp\frac1{[s]_q}(\ttg_{1,\pm s}+\cdots+\ttg_{n,\pm s})$,
$\sX =[\ttx_{n-1,0}^-,[\ttx_{n-2,0}^-,\cdots,
[\ttx_{2,0}^-,\ttx_{1,1}^-]_{\ttv^{-1}}\cdots
]_{\ttv^{-1}}]_{\ttv^{-1}}$ and
$\sY =[\cdots[[\ttx_{1,-1}^+,\ttx_{2,0}^+]_\ttv,\ttx_{3,0}^+]_\ttv,
 \cdots,\ttx_{n-1,0}^+]_\ttv.$
\end{Lem}

We now review the classification theorem of finite dimensional irreducible polynomial $\afUglC$-modules. We first need to introduce the elements $\ms
Q_{i,s}\in\afUglC$, which will be used to define pseudo-highest weight modules.
For $1\leq i\leq n$ and $s\in\mbz$, define the elements $\ms
Q_{i,s}\in\afUglC$ through the generating functions
\begin{equation*}
\begin{split}
&\quad\qquad\ms Q_i^\pm(u):=\exp\bigg(-\sum_{t\geq
1}\frac{1}{[t]}g_{i,\pm t} (\ttv u)^{\pm t}\bigg)=\sum_{s\geq
0}\ms Q_{i,\pm s} u^{\pm s}\in\afUglC[[u,u^{-1}]].
\end{split}
\end{equation*}

For a representation $V$ of $\afUglC$,
a nonzero vector $w\in V$ is called a {\it pseudo-highest weight
vector}\index{pseudo-highest weight vector} if there exist
some $Q_{i,s}\in\mbc$ such that
\begin{equation}\label{HWvector}
\ttx_{j,s}^+w=0,\quad\ms Q_{i,s}w=Q_{i,s}w,\quad
\ttk_iw=\ttv^{\la_i}w
\end{equation}
for all $1\leq i\leq n$ and $1\leq j\leq n-1$ and $s\in\mbz$. The
module $V$ is called a {\it pseudo-highest weight module} if $V=\afUglC w$
for some pseudo-highest weight vector $w$.
We also write the short form $\ms Q_i^\pm(u)w=Q_i^\pm(u)w$ for the
relations $\ms Q_{i,s}w=Q_{i,s}w\,(s\in\mbz)$, where
\begin{equation*}\label{power series Q}
Q_i^\pm(u)=\sum_{s\geq 0}Q_{i,\pm s}u^{\pm s}.
\end{equation*}

Let $V$ be a finite dimensional polynomial representation of $\afUglC$ of type
1. Then $V=\oplus_{\la\in\mbn^{n}}V_\la$, where
$$V_\la=\{x\in V\mid \ttk_jx=\ttv^{\la_j}x, 1\leq j\leq n \},$$
and, since all $\ms
Q_{i,s}$ commute with the $ \ttk_j$,  each $V_\la$ is a direct sum
of generalized eigenspaces of the form
\begin{equation}\label{geigenspace}
V_{\la,\gamma}=\{x\in V_\la\mid (\ms
Q_{i,s}-\gamma_{i,s})^px=0\text{ for some $p$}\, (1\leq i\leq
n,s\in\mbz)\},
\end{equation}
 where $\gamma=(\gamma_{i,s})$ with $\gamma_{i,s}\in\mbc$. Let $\Gamma_i^\pm(u)=\sum_{s\geq 0}\gamma_{i,\pm s} u^{\pm
s}$.

A finite dimensional $\afUglC$-module $V$ is called a {\it
polynomial representation}\index{polynomial representation} if the restriction of $V$ to $\UglC$ is a polynomial representation
of type 1 and, for every weight $\la=(\la_1,\ldots,\la_n)\in\mbn^n$
of $V$, the formal power series $\Gamma_i^\pm(u)$ associated to the
eigenvalues $(\gamma_{i,s})_{s\in\mbz}$ defining the generalized
eigenspaces $V_{\la,\gamma}$ as given in \eqref{geigenspace}, are
polynomials in $u^\pm$ of degree $\la_i$ so that the zeroes of the
functions $\Gamma_i^+(u)$ and $\Gamma_i^-(u)$ are the same.

Following \cite{FM}, an $n$-tuple of polynomials
$\bfQ=(Q_1(u),\ldots,Q_n(u))$ with constant terms $1$ is called {\it
dominant} if, for each $1\leq i\leq n-1$, the ratio
$Q_i(\ttv^{i-1}u)/Q_{i+1}(\ttv^{i+1}u)$ is a polynomial. Let
$\sQ(n)$ be the set of dominant $n$-tuples of polynomials.

For $g(u)=\prod_{1\leq i\leq m}(1-a_iu)\in\mbc[u]$
with constant term $1$ and $a_i\in\mbc^*$, define
\begin{equation}\label{f^pm(u)}
g^\pm(u)=\prod_{1\leq i\leq m}(1-a_i^{\pm1}u^{\pm1}).
\end{equation}
For $\bfQ=(Q_1(u),\ldots,Q_{n}(u))\in\sQ(n)$, define
$Q_{i,s}\in\mbc$, for $1\leq i\leq n$ and $s\in\mbz$, by the
following formula
$$Q_i^\pm(u)=\sum_{s\geq 0}Q_{i,\pm s}u^{\pm s},$$
where $Q_i^\pm(u)$ is defined using \eqref{f^pm(u)}. Let $I(\bfQ)$
be the left ideal of $\afUglC$ generated by $\ttx_{j,s}^+ ,\quad\ms
Q_{i,s}-Q_{i,s},$ and $\ttk_i-\ttv^{\la_i}$, for $1\leq j\leq n-1$,
$1\leq i\leq n$ and $s\in\mbz$, where $\la_i=\mathrm{deg}Q_i(u)$,
and define
$$M(\bfQ)=\afUglC/I(\bfQ).$$
Then $M(\bfQ)$ has a unique irreducible quotient, denoted by $L(\bfQ)$.
The polynomials $Q_i(u)$ are called {\it Drinfeld
polynomials} associated with $L(\bfQ)$.

\begin{Thm}[\cite{FM}]\label{classification of simple afUglC-modules}
The $\afUglC$-modules $L(\bfQ)$ with $\bfQ\in\sQ(n)$ are all
nonisomorphic finite dimensional irreducible polynomial representations
of $\afUglC$.
\end{Thm}

If $\bfQ,\bfQ'\in\sQ(n)$ is such that $Q_j(\ttv^{j-1}u)/Q_{j+1}(\ttv^{j+1}u)=Q_j'(\ttv^{j-1}u)/Q_{j+1}'(\ttv^{j+1}u)$ and $\deg Q_j(u)-\deg Q_{j+1}(u)=\deg Q_j'(u)-\deg Q_{j+1}'(u)$ for $1\leq j\leq n-1$, then by \cite[4.7.1 and 4.7.2]{DDF}, we have $L(\bfQ)|_{\afUslC}\cong L(\bfQ')|_{\afUslC}$.
Thus we may denote $L(\bfQ)|_{\afUslC}$ by $\bar L(\bfP)$,
where
$\bfP=(P_1(u),\ldots,P_{n-1}(u))$ with
$P_j(u)=Q_j(\ttv^{j-1}u)/Q_{j+1}(\ttv^{j+1}u)$.

 Let $\sP(n)$ be the set
of $(n-1)$-tuples of polynomials with constant terms $1$.
The following result is due to Chari--Pressley (cf.
\cite{CP91,CPbk,CP95}).

\begin{Thm}
The modules $\Lcp(\bfP)$ with $\bfP\in\sP(n)$ are all nonisomorphic
finite dimensional irreducible $\afUslC$-modules of  type $1$.
\end{Thm}

\section{Affine quantum Schur algebras}
In this section we collect some facts about
extended affine Hecke algebras and affine quantum Schur algebras, which will be used in \S 4.
The extended affine Hecke algebra $\afHrC$ is defined to be the algebra generated by
$$T_i,\quad X_j^{\pm 1}(\text{$1\leq i\leq r-1$, $1\leq j\leq r$}),$$
 and relations
$$\aligned
 & (T_i+1)(T_i-\ttv^2)=0,\\
 & T_iT_{i+1}T_i=T_{i+1}T_iT_{i+1},\;\;T_iT_j=T_jT_i\;(|i-j|>1),\\
 & X_iX_i^{-1}=1=X_i^{-1}X_i,\;\; X_iX_j=X_jX_i,\\
 & T_iX_iT_i=\ttv^2 X_{i+1},\;\;  X_jT_i=T_iX_j\;(j\not=i,i+1).
\endaligned$$

Let $\fSr$ be the symmetric group with generators $s_i:=(i,i+1)$ for $1\leq i\leq r-1$.
Let $I(n,r)=\{(i_1,\ldots,i_r)\in\mbz^r\mid 1\leq i_k\leq
n,\,\forall k\}.$
The symmetric group $\fSr$ acts
on the set $I(n,r)$ by place permutation:
\begin{equation*}\label{place permutation}
\bfi w=(i_{w(k)})_{k\in\mbz},\quad\text{
for $\bfi\in I(n,r)$ and $w\in\fSr$.}
\end{equation*}

Let $\OgC$ be a vector space over $\mathbb C$ with basis $\{\og_i\mid i\in\mathbb Z\}$.  For
 $\bfi=(i_1,\ldots,i_r)\in\mbz^r$, write
$$\og_\bfi=\og_{i_1}\ot\og_{i_2}\ot\cdots\ot \og_{i_r}=\og_{i_1}\og_{i_2}\cdots \og_{i_r}\in\OgC^{\ot r}.$$
The tensor space $\OgC^{\ot r}$
admits a right $\afHrC$-module structure defined by
\begin{equation*}\label{afH action}
\begin{cases}
\og_{\bf i}\cdot X_t^{-1}
=\og_{i_1}\cdots\og_{i_{t-1}}\og_{i_t+n}\og_{i_{t+1}}\cdots\og_{i_r},\qquad \text{ for all }\bfi\in \mbz^r;\\
{\og_{\bf i}\cdot T_k=\left\{\begin{array}{ll} \ttv^2\og_{\bf
i},\;\;&\text{if $i_k=i_{k+1}$;}\\
\ttv\og_{\bfi s_k},\;\;&\text{if $i_k<i_{k+1}$;}\qquad\text{ for all }\bfi\in I(n,r),\\
\ttv\og_{\bfi s_k}+(\ttv^2-1)\og_{\bf i},\;\;&\text{if
$i_{k+1}<i_k$,}
\end{array}\right.}
\end{cases}
\end{equation*}
where $1\leq k\leq r-1$ and $1\le t\le r$.

The algebra $$\afSrC:=\End_{\afHrC}(\sT_\vtg(n,r))$$
is called an affine $q$-Schur algebra, where $\sT_\vtg(n,r)=\OgC^{\ot r}$. Let $\OgnC$ be the subspace of $\OgC$ spanned by $\og_i$ with $1\leq i\leq n$ and $\HrC$ be the subalgebra of $\afHrC$ generated by $T_k$ for $1\leq k\leq r-1$. Then the algebra $\SrC: =\End_{\HrC}(\sT(n,r))$ is called a $q$-Schur algebra, where $\sT(n,r)=\OgnC^{\ot r}$.

The algebras $\afUglC$ and  $\afSrC$ are related by an algebra homomorphism $\zeta_r$, which we now describe.
For $i\in\mbz$, let $\bar i$ denote the integer modulo $n$.
The complex vector space $\OgC$ is a natural $\DC(n)$-module with the action
\begin{equation} \label{QGKMAlg-action}
\aligned
E_i\cdot \og_s&=\dt_{\ol{i+1},\bar s}\og_{s-1},\quad F_i\cdot \og_s=\dt_{\bar i,\bar
s}\og_{s+1},\quad
K_i^{\pm 1}\cdot \og_s=\ttv^{\pm\dt_{\bar i,\bar s}}\og_s,\\
&\sfz_t^+\cdot\og_s=\og_{s-tn},\quad\text{and }\;\;
\sfz_t^-\cdot\og_s=\og_{s+tn}.
\endaligned
\end{equation}
The Hopf algebra structure induces a $\DC(n)$-module $\OgC^{\ot r}$. By \cite[3.5.5]{DDF}, the actions of $\DC(n)$ and $\afHrC$ on $\OgC^{\ot r}$ are commute. We will identify $\DC(n)$ and $\afUglC$ via the algebra isomorphism $f$ defined in \ref{DDFIsoThm}. Consequently, there is an algebra homomorphism
\begin{equation*}\label{afzrC}
\zrC:\afUglC=\DC(n)\lra\afSrC.
\end{equation*}
It is proved in \cite[3.8.1]{DDF} that $\zrC$ is surjective.
Let $\UglC$ be the subalgebra of $\DC(n)$ generated by $E_i,\ F_i,\  K_j,\ K_j^{-1}$ for  $1\leq i\leq n-1$ and $1\leq j\leq n$.
The restriction of $\zrC$ to $\UglC$ induces a surjective algebra homomorphism
$\zrC:\UglC\lra\SrC$ (cf. \cite{Ji}).
Every $\afSrC$-module (resp., $\SrC$-module) will be inflated into a $\afUglC$-module (resp., $\UglC$-module)
via $\zrC$.

 The following easy lemma relates $\OgC^{\ot r}$ with $\OgnC^{\ot r}$.

\begin{Lem}[{\cite[4.1.1]{DDF}}]\label{Lem1}
There is a $\UglC$-$\afHrC$-bimodule isomorphism
$$\OgnC^{\ot r}\ot_{\HrC}\afHrC\stackrel{\sim}{\lra}\OgC^{\ot r},\;
x\ot h\longmapsto xh.$$
\end{Lem}

The irreducible $\afHrC$-modules were classified in \cite{Zelevinsky,Rogawski}, which we now describe.
 For
$\bfa=(a_1,\ldots,a_r)\in(\mbc^*)^r$, let $M_{\bfa}=\afHrC/J_\bfa$,
where $J_\bfa$ is the left ideal of $\afHrC$ generated by $X_j-a_j$
for $1\leq j\leq r$.

A {\it segment} $\sfs$ with center $a\in\mbc^*$ is by definition an
ordered sequence
$$\sfs=(a\ttv^{-k+1},a\ttv^{-k+3},\ldots,a\ttv^{k-1})\in(\mbc^*)^k.$$
 Here $k$ is called the length
of the segment, denoted by $|\sfs|$. If
$\bfs=\{\sfs_1,\ldots,\sfs_p\}$ is an unordered collection of
segments, define $\wp(\bfs)$ to be the partition associated with the
sequence $(|\sfs_1|,\ldots,|\sfs_p|)$. That is,
$\wp(\bfs)=(|\sfs_{i_1}|,\ldots,|\sfs_{i_p}|)$ with
$|\sfs_{i_1}|\geq\cdots\geq|\sfs_{i_p}|$, where
$|\sfs_{i_1}|,\ldots,|\sfs_{i_p}|$ is a permutation of
$|\sfs_1|,\ldots,|\sfs_p|$. We also call $|\bfs|:=|\wp(\bfs)|$ the
length of $\bfs$.

Let $\mathscr S_r$ be the set of unordered collections of segments
$\bfs$ with $|\bfs|=r$. Then $\mathscr
S_r=\cup_{\mu\in\Lambda^+(r)}\mathscr S_{r,\mu}$, where $\mathscr
S_{r,\mu}=\{\bfs\in\mathscr S_r\mid\wp(\bfs)=\mu\}$ and $\Lambda^+(r)$ is the set of partitions of $r$.
\index{$\Lambda^+(r)$, set of partitions of $r$}

If $w=s_{i_1}s_{i_2}\cdots s_{i_m}$ is reduced let $T_w=T_{i_1}T_{i_2}\cdots T_{i_m}$.
For $p\geq 1$ let
\begin{equation}\label{Lanr}
\La(p,r)=\{\mu\in\mbn^p\mid\sum_{1\leq i\leq p}\mu_i=r\}
\end{equation}
For $\mu\in\La(p,r)$ let $\frak S_\mu$ be the corresponding standard Young subgroup of the symmetric group $\frak S_r$, and let $\msD_\mu=\{d\in\frak S_r\mid\ell(wd)=\ell(w)+\ell(d)\ \text{for}\ w\in\frak S_\mu\}$.
For $\mu\in\La(p,r)$ let
\begin{equation}\label{Imu}
\sI_\mu=\sH(r)_\mbc
y_\mu,
\end{equation}
 where
$y_{\mu}=\sum_{w\in\fS_\mu}(-\ttv^2)^{-\ell(w)}T_w\in\HrC$.
For $\bfs=\{\sfs_1,\ldots,\sfs_p\}\in\mathscr S_{r,\mu}$, let
$\bfa(\bfs)=(\sfs_1,\ldots,\sfs_p)\in(\mbc^*)^r$
be the $r$-tuple obtained by juxtaposing the segments in $\bfs$.
Let $\iota :\HrC\ra M_{\bfa(\bfs)}$ be the natural $\HrC$-module isomorphism defined by sending $h$ to $\bar h$. Let $$ \bar\sI_\mu=\iota(\sI_\mu)=\sH(r)_\mbc
\bar y_\mu=\afHrC\bar y_\mu.$$
Then,
\begin{equation}\label{signed permutation module}
\sH(r)_\mbc y_\mu\cong E_\mu\oplus(\bigoplus_{\nu\vdash r,
\nu\rhd\la}m_{\nu,\mu}E_\nu),
\end{equation}
where $E_\nu$ is the left cell module defined by the
Kazhdan--Lusztig's C-basis \cite{KL79} associated with the left cell
containing $w_{0,\nu}$.

Let $V_\bfs$ be the unique composition factor of the $\afHrC$-module
$\afHrC \bar y_\mu$ such that the multiplicity of $E_{\mu}$ in
$V_\bfs$ as an $\sH(r)_\mbc$-module is nonzero.

The following classification theorem is due to
Zelevinsky \cite{Zelevinsky} and Rogawski \cite{Rogawski}.

\begin{Thm}\label{classification irr affine Hecke algebra}
The modules $ V_\bfs$ with $\bfs\in\mathscr S_r$ are all nonisomorphic finite dimensional
irreducible $\afHrC$-modules.
\end{Thm}

Let $\afSrC\hmod$ (resp., $\afHrC\hmod$) be the category of finite dimensional $\afSrC$-modules (resp., $\afHrC$-modules). The categories $\afSrC\hmod$ and $\afHrC\hmod$ are related by the Schur functor $\sF$, which we now define.
Using the $\afSrC$-$\afHrC$-bimodule $\OgC^{\ot r}$, we define a functor
\begin{equation}\label{functor sF}
\sF=\sF_{n,r}:\afHrC\hmod\lra\afSrC\hmod,\;V\longmapsto\OgC^{\ot
r}\ot_{\afHrC}V.
\end{equation}

Let $$\ms S_r^{(n)}=\{\bfs=
\{\sfs_1,\ldots,\sfs_p\}\in\ms S_r,\,p\geq 1,\,|\sfs_i|\leq n,\,\forall i\}.$$
The following classification theorem is given in \cite[4.3.4 and 4.5.3]{DDF}.

\begin{Lem}\label{classification 1}
For $\bfs\in\mathscr S_r$
we have $\sF(V_\bfs)\neq0$
if and only if $\bfs\in\mathscr S_{r}^{(n)}$. Furthermore, the set
 $$\{\sF(V_\bfs)\mid\bfs\in\mathscr S_{r}^{(n)}\}$$
is a complete set of nonisomorphic finite dimensional irreducible
$\sS_\vtg(n,r)_\mbc$-modules.
\end{Lem}

The following result, which will be used in \ref{main theorem},
is taken from \cite[7.6]{CP96} and \cite[4.4.2 and 4.6.5]{DDF}.
\begin{Lem}\label{n geq r}
Assume $n\geq r$. Let $\bfs=(a\ttv^{-r+1},a\ttv^{-r+3},\cdots,a\ttv^{r-1})$ be a single segment and $\mu=\wp(\bfs)=(r)$. Then $V_\bfs=\bar\sI_\mu$ and
$\sF(V_\bfs)\cong L(\bfQ),$
where $\bfQ=(Q_1(u),\cdots,Q_n(u))$ with $Q_n(u)=(1-a\ttv^{-n+1}u)^{\dt_{n,r}}$ and $\frac{Q_i (u\ttv^{i-1})} {Q_{i+1} (u\ttv^{i+1})}=(1-au)^{\dt_{i,r}}$ for $1\leq i\leq n-1$.
\end{Lem}

\section{Identification of irreducible $\afSrC$-modules}
In this section we will prove that $\sF(\bar\sI_{\wp(\bfs)})$ is isomorphic to the tensor product of irreducible $\afSrC$-modules for $\bfs\in\ms S_r^{(n)}$ and $\sF(\bar\sI_{\wp(\bfs)})=0$ for $\bfs\not\in\ms S_r^{(n)}$ in \ref{standard module}. Using this result, we will relate the parametrization of irreducible $\afHrC$-modules, via the functor $\sF$ defined in \eqref{functor sF}, to the parametrization of finite dimensional irreducible polynomial representations of $\afUglC$ in \ref{main theorem}. As applications, we will classify finite dimensional irreducible $\afSrC$-modules in
\ref{classification 2}, and generalize \cite[(6.5f)]{Gr80} to the affine case.

To compute $\sF(\bar\sI_{\wp(\bfs)})$, we need  a result of Rogawski \cite[4.3]{Rogawski}, which we now describe.
For $1\leq j\leq p$, let $\sH_{\mu,j}$ be the subalgebra of $\HrC$ generated by $T_i$ with $s_i\in\fS_{\mu^{(j)}}$,
where $$\mu^{(j)}=(1^{\mu_{[1,j-1]}},\mu_j,1^{r-\mu_{[1,j]}}),$$
and $\mu_{[1,j]}=\mu_1+\mu_2+\cdots+\mu_{j}$.
Since $\sH_{\mu,j}\cong\sH(\mu_j)_\mbc$ for $1\leq j\leq p$ and $\OgnC^{\ot\mu_j}$ is a right $\sH(\mu_j)_\mbc$-module, $\OgnC^{\ot\mu_j}$ can be also regarded as a right $\sH_{\mu,j}$-module.

Recall the notation $\sI_\mu$ defined in \eqref{Imu}.
For $\mu\in\La(p,r)$ and $1\leq j\leq p$  let $$\sJ_\mu=
\bigcap_{s_i\in\fS_{\mu}\atop 1\leq i\leq r-1}\HrC C_i,
\quad  \sJ_{\mu,j}=\bigcap_{s_i\in\fS_{\mu^{(j)}}
\atop 1\leq i\leq r-1}\sH_{\mu,j} C_i\quad\text{and}\quad\sI_{\mu,j}=\sH_{\mu,j}y_{\mu^{(j)}}.$$
where $C_i=\ttv^{-1}T_i-\ttv$ and
$y_{\mu^{(j)}}=\sum_{w\in\fS_{\mu^{(j)}}}(-\ttv^2)^{-\ell(w)}T_w$.
 By \cite[4.3]{Rogawski} we have the following result.
\begin{Lem}\label{Rogawski}
We have $\sI_\mu=\sJ_\mu$, $\sI_{\mu,j}=\sJ_{\mu,j}$ for $\mu\in\La(p,r)$ and $1\leq j\leq p$.
\end{Lem}

\begin{Lem}\label{Lem3}
Assume $I$ is a left ideal of $\HrC$. Then $\OgnC^{\ot r}\ot_{\HrC}I\cong\OgnC^{\ot r}I$.
\end{Lem}
\begin{proof}
Since $\HrC$ is semisimple, there exist a left ideal $J$ of $\HrC$ such that $\HrC=I\op J$. Then $\OgnC^{\ot r}\cong \OgnC^{\ot r}\ot_{\HrC}\HrC\cong \OgnC^{\ot r}\ot_{\HrC}I\op \OgnC^{\ot r}\ot_{\HrC}J$. Thus the natural linear map $f:\OgnC^{\ot r}\ot_{\HrC}I\ra\OgnC^{\ot r}$ defined by sending $w\ot h$ to $wh$ is injective. Consequently, $\OgnC^{\ot r}\ot_{\HrC}I\cong Im(f)=\OgnC^{\ot r}I$.
\end{proof}

By \ref{Lem1}, \ref{Rogawski} and \ref{Lem3} we conclude that $\sF(\bar\sI_\mu)\cong\OgnC^{\ot r}\ot_{\HrC}\bar\sJ_\mu\cong\OgnC^{\ot r}\sJ_\mu$, where $\mu=\wp(\bfs)$ for some $\bfs\in\ms S_r$. We now compute $\OgnC^{\ot r}\sJ_\mu$.

\begin{Lem}\label{Lem2}
For $\mu\in\La(p,r)$, we have $$\OgnC^{\ot r}\sJ_\mu
= \OgnC^{\ot\mu_1}
\sJ_{\mu,1}
\ot\cdots\ot \OgnC^{\ot\mu_p}
\sJ_{\mu,p}.$$
\end{Lem}
\begin{proof}
Since $\sJ_\mu=\cap_{1\leq j\leq p} \sJ_{\mu^{(j)}}$ we have $\OgnC^{\ot r}\sJ_\mu\han
\bigcap_{1\leq j\leq p}\big(\OgnC^{\ot r} \sJ_{\mu^{(j)}}\big)$. Furthermore by \ref{Rogawski}  we have $\sJ_{\mu^{(j)}}=\sI_{\mu^{(j)}}=\sX_{\mu,j}\sI_{\mu,j}
=\sX_{\mu,j}\sJ_{\mu,j}$ where $\sX_{\mu,j}=\spann\{T_w\mid w\in\msD_{\mu^{(j)}}^{-1}\}$. This implies that $$\OgnC^{\ot r}  \sJ_{\mu^{(j)}}=
\OgnC^{\ot r}\sJ_{\mu,j}=\OgnC^{\mu_1}\ot\cdots\ot\OgnC^{\mu_{j-1}}
\ot\OgnC^{\ot \mu_j}\sJ_{\mu_j}\ot\OgnC^{\ot\mu_{j+1}}\ot\cdots\ot
\OgnC^{\ot\mu_p}$$ for $1\leq j\leq p$. Thus,
$$\OgnC^{\ot r}\sJ_\mu\han
\bigcap_{1\leq j\leq p}\big(
\OgnC^{\mu_1}\ot\cdots\ot\OgnC^{\mu_{j-1}}
\ot\OgnC^{\ot \mu_j}\sJ_{\mu_j}\ot\OgnC^{\ot\mu_{j+1}}\ot\cdots\ot
\OgnC^{\ot\mu_p}
\big)=
\OgnC^{\ot\mu_1}
\sJ_{\mu,1}
\ot\cdots\ot \OgnC^{\ot\mu_p}
\sJ_{\mu,p}.$$ On the other hand, we assume $w_1h_1\ot\cdots\ot w_ph_p\in\OgnC^{\ot\mu_1}
\sJ_{\mu,1} \ot\cdots\ot \OgnC^{\ot\mu_p} \sJ_{\mu,p}$, where  $w_j\in\OgnC^{\ot\mu_j}$
and $h_j\in \sJ_{\mu,j}$. Since $h_kh_l=h_lh_k$ for any $k,l$ and $h_j\in \sJ_{\mu,j}$, we have $h_1h_2\cdots h_p=(h_1\cdots h_{j-1}h_{j+1}\cdots h_p)h_j\in \HrC \sJ_{\mu,j}\han \HrC C_i$ for $1\leq i\leq r-1$, $1\leq j\leq p$  with $s_i\in\frak S_{\mu^{(j)}}$. This implies that $h_1h_2\cdots h_p\in\sJ_\mu$. It follows that $w_1h_1\ot\cdots\ot w_ph_p=(w_1 \ot\cdots\ot w_p)h_1\cdots h_p\in \OgnC^{\ot r}\sJ_\mu$.
The assertion follows.
\end{proof}

For $\mu\in\La(p,r)$ and $1\leq j\leq p$, let $\ti\sH_{\mu,j}$ be the subalgebra of $\afHrC$ generated by $T_i$ and $X_{\mu_{[1,j-1]}+1},\cdots,X_{\mu_{[1,j]}}$ with $s_i\in\fS_{\mu^{(j)}}$. Since $\ti\sH_{\mu,j}\cong\sH_\vtg(\mu_j)_\mbc$ and $\OgC^{\ot\mu_j}$ is a right $\sH_\vtg(\mu_j)_\mbc$-module, $\OgC^{\ot\mu_j}$ can be regarded as a right $\ti\sH_{\mu,j}$-module.

For $\bfs=\{\sfs_1,\ldots,\sfs_p\}\in\mathscr S_{r,\mu}$, let
$\bfa=(\sfs_1,\ldots,\sfs_p)\in(\mbc^*)^r$
be the $r$-tuple obtained by juxtaposing the segments in $\bfs$.
For $1\leq j\leq p$ let $\frak I_{\mu,j}$ be the left ideal of $\ti\sH_{\mu,j}$ generated by $X_k-a_k$ for $\mu_{[1,j-1]}+1\leq k\leq\mu_{[1,j]}$. Let $\iota_j:\sH_{\mu,j}\ra\ti\sH_{\mu,j}/\frak I_{\mu,j}$ be the natural $\sH_{\mu,j}$-module isomorphism defined by sending $h$ to $\bar h$. Let
$$ \bar\sI_{\mu,j}=\iota_j(\sI_{\mu,j})=\sH_{\mu,j}\bar y_{\mu^{(j)}}
=\ti\sH_{\mu,j}\bar y_{\mu^{(j)}}.$$

By \ref{Lem2} we have the following corollary.
\begin{Coro}\label{vi}
Maintain the notation above.
There is a $\UglC$-module isomorphism $$\vi:(\OgC^{\ot\mu_1}\ot_{\ti\sH_{\mu,1}}\!\bar \sI_{\mu,1})\ot\cdots\ot(\OgC^{\ot\mu_p}\ot_{\ti\sH_{\mu,p}}\!\bar \sI_{\mu,p})\ra\sF(\bar \sI_\mu)$$ such that $\vi(w_1\ot\ol{h_1}\ot\cdots\ot w_p\ot\ol{h_p})=w_1\ot\cdots\ot w_p\ot\ol{h_1\cdots h_p}$ for $w_j\in\OgnC^{\ot\mu_j}$ and $h_j\in\sI_{\mu,j}$ with $1\leq j\leq p$.
\end{Coro}
\begin{proof}
Combining \ref{Lem1}, \ref{Rogawski} with \ref{Lem3} yields $\sF(\bar\sI_\mu)\cong\OgnC^{\ot r}\ot_{\HrC}\bar\sJ_\mu\cong\OgnC^{\ot r}\sJ_\mu$
and $\OgC^{\ot\mu_j}\ot_{\ti\sH_{\mu,j}}\bar\sI_{\mu,j}
\cong\OgnC^{\ot^{\mu_j}}\ot_{\sH_{\mu,j}}\bar\sJ_{\mu,j}
\cong\OgnC^{\ot^{\mu_j}} \sJ_{\mu,j}$ for $1\leq j\leq p$. This, together with \ref{Lem2}, implies the assertion.
\end{proof}

We now prove that $\vi$ is in fact a $\afUglC$-module isomorphism.
\begin{Lem}\label{isomorphism vi}
The map $\vi$ is a $\afUglC$-module homomorphism.
\end{Lem}
\begin{proof}
Let $u\in\afUglC$ and $w=w_1\ot\ol{h_1}\ot\cdots\ot w_p\ot\ol{h_p}\in
(\OgC^{\ot\mu_1}\ot_{\ti\sH_{\mu,1}}\!\bar \sI_{\mu,1})\ot\cdots\ot(\OgC^{\ot\mu_p}\ot_{\ti\sH_{\mu,p}}\!\bar \sI_{\mu,p})$, where $w_i\in\OgnC^{\ot\mu_i}$ and $h_i\in\sI_{\mu,i}$ for $1\leq i\leq p$.
Assume $\Delta^{(p-1)}(u)=\sum_{(u)}u_1\ot\cdots\ot u_p$, $u_iw_i=\sum_{k_i}w_{i,k_i}g_{i,k_i}$ and $g_{i,k_i}h_i=\sum_{j_i}g_{i,k_i,j_i}X_{j_i}$, where $w_{i,k_i}\in\OgnC^{\ot\mu_i}$, $g_{i,k_i}\in\ti\sH_{\mu,i}$, and $g_{i,k_i,j_i}\in\sH_{\mu,i}$, $X_{j_i}\in\ti\sH_{\mu,i}$. Then $$g_{i,k_i}(\iota_i(h_i))=g_{i,k_i}\ol{h_i}=\sum_{j_i}a_{j_i}
\ol{g_{i,k_i,j_i}}.$$
Hence,
\begin{equation*}
\begin{split}
uw&
= \sum_{(u)}u_1w_1\ot\ol{h_1}\ot\cdots\ot u_pw_p\ot\ol{h_p}
\\
&= \sum_{(u)}\sum_{k_1,\cdots,k_p}w_{1,k_1}\ot g_{1,k_1}\ol{h_1}\ot\cdots\ot w_{p,k_p}\ot g_{p,k_p}\ol{h_p} \\
&= \sum_{(u)}\sum_{k_1,\cdots,k_p\atop j_1,\cdots,j_p}a_{j_1}\cdots a_{j_p}w_{1,k_1}\ot \ol{g_{1,k_1,j_1}}\ot\cdots\ot w_{p,k_p}\ot\ol{g_{p,k_p,j_p}}.
\end{split}
\end{equation*}
Since $g_{1,k_1}\cdots g_{p,k_p}\ol{h_1\cdots h_p} =\ol{
g_{1,k_1}h_1\cdots g_{p,k_p} h_p}=\sum_{j_1,\cdots,j_p}a_{j_1}\cdots a_{j_p}\ol{g_{1,k_1,j_1}\cdots g_{p,k_p,j_p}}$,
we conclude that
\begin{equation*}
\begin{split}
\vi(uw)&=\sum_{(u)}\sum_{k_1,\cdots,k_p\atop j_1,\cdots,j_p}a_{j_1}\cdots a_{j_p}w_{1,k_1}\ot  \cdots\ot w_{p,k_p}\ot\ol{g_{1,k_1,j_1}\cdots g_{p,k_p,j_p}}\\
&  = \sum_{(u)}\sum_{k_1,\cdots,k_p}w_{1,k_1}\ot\cdots\ot w_{p,k_p}\ot g_{1,k_1}\cdots g_{p,k_p}\ol{h_1\cdots h_p}\\
& = \sum_{(u)}u_1w_{1}\ot\cdots\ot u_pw_{p}\ot \ol{h_1\cdots h_p}\\
 &=u(w_1\ot\cdots\ot w_p\ot\ol{h_1\cdots h_p})\\
&=u\vi(w).
\end{split}
\end{equation*}
The proof is completed.
\end{proof}

We can now describe $\sF(\bar\sI_{\wp(\bfs)})$ as follows.
\begin{Prop}\label{standard module}
Let $\bfs=\{\sfs_1,\ldots,\sfs_p\}\in\ms S_{r,\mu}$. Then  $\sF(\bar\sI_\mu)=0$ for $\bfs\not\in\ms S_{r}^{(n)}$ and
$\sF(\bar \sI_\mu)\cong L(\bfQ_1)\ot\cdots\ot L(\bfQ_p)$ for $\bfs\in\ms S_{r}^{(n)}$, where $\bfQ_i=(Q_{i,1}(u),\cdots,Q_{i,n}(u))$ with $Q_{i,n}(u)=(1-a_i\ttv^{-n+1}u)^{\dt_{\mu_i,n}}$ and
$\frac{Q_{i,j} (u\ttv^{j-1})} {Q_{i,j+1} (u\ttv^{j+1})}=(1-a_iu)^{\dt_{j,\mu_i}}$ for $1\leq i\leq p$ and $1\leq j\leq n-1$.
\end{Prop}
\begin{proof}
Since $\bar\sI_{\mu_i}\cong V_{\sfs_i}$ for $1\leq i\leq p$, by \ref{vi} and \ref{isomorphism vi} we conclude that
$\sF(\bar\sI_\mu)=\sF_{n,r}(\bar\sI_\mu)\cong\sF_{n,\mu_1}(V_{\sfs_1})
\ot\cdots\ot\sF_{n,\mu_p}(V_{\sfs_p})$. If $\bfs\not\in\ms S_{r}^{(n)}$, then there exist $1\leq k\leq p$ such that $|\sfs_k|=\mu_k>n$. By \ref{classification 1} we have $\sF_{n,\mu_k}(V_{\sfs_k})=0$ and hence $\sF(\bar\sI_\mu)=0$. If $\bfs \in\ms S_{r}^{(n)}$, then by \ref{n geq r} we have  $\sF_{n,\mu_i}(V_{\sfs_i})\cong L(\bfQ_i)$ for $1\leq i\leq p$. Consequently, $\sF(\bar\sI_\mu)\cong  L(\bfQ_1)\ot\cdots\ot L(\bfQ_p)$.
\end{proof}

We now turn to studying $\sF(V_\bfs)$ for $\bfs\in\ms S_r^{(n)}$. To compute $\sF(V_\bfs)$, we need to generalize \cite[7.2]{CP96} to the case of $n\leq r$.
Recall the notation $\Lanr$ defined in \eqref{Lanr}.
Let $\La^+(n,r)=\La(n,r)\cap\La^+(r)$. For $\la\in\mbnn$ let $L(\la)$ be the irreducible $\UglC$-module with highest weight $\la$. For $1\leq i\le n$, let $\bffkk_i=\zeta_r(K_i)$ and
$$\bigg[ {\bffkk_i;0 \atop t} \bigg] =
\prod_{s=1}^t \frac
{\bffkk_i\ttv^{-s+1}-\bffkk_i^{-1}\ttv^{s-1}}{\ttv^s-\ttv^{-s}}.$$
For $\mu\in\mbnn$ let $\bffkk_\mu=\big[{\bffkk_1;0\atop \mu_1}\big]
\cdots\big[{\bffkk_n;0\atop \mu_n}\big]$. The following result
is the generalization of \cite[7.2]{CP96}.

\begin{Lem}\label{finite case}
Let $\mu\in\La^+(r)$. Then $\OgnC^{\ot r}\ot_{\HrC}E_\mu\not=0$ if and only if $\mu'\in\La(n,r)$, where $\mu'$ is the dual partition of $\mu$. Furthermore if $\mu'\in\La^+(n,r)$, then $\OgnC^{\ot r}\ot_{\HrC} E_\mu\cong L(\mu')$.
\end{Lem}
\begin{proof}
We choose $N$ such that $N>\max\{n,r\}$. Let $e=\sum_{\mu\in\La(n,r)}\bffkk_\mu\in\sS(N,r)_\mbc$. It is well known
that for $\mu\in\La^+(N,r)$, $eL(\mu)\not=0$ if and only if $\mu\in\La(n,r)$ (cf. \cite[6.5(f)]{Gr80}). Furthermore by \cite[4.3.3]{DDF} and \cite[7.2]{CP96} we have $\OgnC^{\ot r}\ot_{\HrC}E_\mu\cong e\big(
\Og_{N,\mbc}^{\ot r}\ot_{\HrC}E_\mu\big)\cong e(L(\mu'))$. Thus $\OgnC^{\ot r}\ot_{\HrC}E_\mu\not=0$ if and only if $\mu'\in\La(n,r)$. If $\mu'\in\La^+(n,r)$, then $\OgnC^{\ot r}\ot_{\HrC}E_\mu \cong e(L(\mu'))\cong L(\mu')$.
\end{proof}

In the case of $n>r$, the Drinfeld polynomials associated with $\sF(V_\bfs)$ were calculated for $\bfs\in\ms S_r^{(n)}$ in \cite[7.6]{CP96} and \cite[4.4.2]{DDF}. We are now prepared to
use \ref{standard module} and \ref{finite case} to generalize [loc. cit.] to the case of $n\leq r$ in \ref{main theorem}.

Let $\sQ(n)_r=\{\bfQ\in\sQ(n)\mid\sum_{1\leq i\leq n}\deg Q_i(u)=r\}$.
For $\bfs=\{\sfs_1,\ldots,\sfs_p\}\in\ms S_r^{(n)}$ with
$$\sfs_i=(a_i\ttv^{-\mu_i+1},a_i\ttv^{-\mu_i+3},\ldots,a_i\ttv^{\mu_i-1})\in(\mbc^*)^{\mu_i},$$
define $\bfQ_\bfs=(Q_1(u),\ldots,Q_n(u))$ by setting  $Q_n(u)=\prod_{1\leq i\leq p\atop\mu_i=n}(1-a_iu\ttv^{-n+1})$ and
$$Q_i(u)=
P_i(u\ttv^{-i+1})P_{i+1}(u\ttv^{-i+2})\cdots
P_{n-1}(u\ttv^{n-2i})Q_n(u\ttv^{2(n-i)})$$ for $1\leq i\leq n-1$, where
$$P_{i}(u)=\prod_{1\leq j\leq p\atop \mu_j=i}(1-a_ju).$$
Then $$\sum_{1\leq i\leq n}\deg Q_i(u)=n\deg Q_n(u)+\sum_{1\leq i\leq n-1}i\deg P_i(u)=\sum_{1\leq i\leq p}\mu_i=r.$$
So $\bfQ_\bfs\in\sQ(n)_r$.  Consequently, we obtain a map
$\pa_{n,r}:\ms S_r^{(n)}\ra\sQ(n)_r$ defined by sending $\bfs$ to $\bfQ_\bfs$.

\begin{Lem}\label{bijective map}
The map
$\pa_{n,r}:\ms S_r^{(n)}\ra\sQ(n)_r$ is bijective.
\end{Lem}
\begin{proof}
It is clear that $\pa_{n,r}$ is injective.
Let $\bfQ=(Q_1(u),\ldots,Q_n(u))\in\sQ(n)_r$ and let $\la\in\La(n,r)$, with $\la_i=\mathrm{deg}
Q_i(u)$. For $1\leq
j\leq n-1$ let
$$
P_j (u) =\frac{Q_j (u\ttv^{j-1})} {Q_{j+1} (u\ttv^{j+1})}
$$
 and $\nu_j={\rm deg\,}
P_j(u)=\la_j-\la_{j+1}$.
We write, for $1\leq i\leq n-1$,
$$P_i(u)=(1-a_{\nu_1 +\cdots+\nu_{i-1}+1}u)(1-a_{\nu_1 +\cdots+\nu_{i-1}+2}u)
\cdots(1-a_{\nu_1 +\cdots+\nu_{i-1}+\nu_i}u),$$
and $Q_n(u)=(1-b_1u)\cdots(1-b_{\la_n}u)$. Let $p'=\sum_{1\leq i\leq n-1}\nu_i$ and $p=p'+\la_n$.
Let $\bfs=\{\sfs_1,\ldots,\sfs_p\}$, where
$$\sfs_i=
\begin{cases}
(a_i\ttv^{-\mu_i+1},a_i\ttv^{-\mu_i+3},\ldots,a_i\ttv^{\mu_i-1})&\text{for
$1\leq i\leq p'$}\\
(b_{i-p'},b_{i-p'}\ttv^2,\cdots,b_{i-p'}\ttv^{2(n-1)})&\text{for $p'+1\leq i\leq p$}
\end{cases}
$$
and $(\mu_1,\ldots,\mu_{p'})=(1^{\nu_1},\ldots,(n-1)^{\nu_{n-1}})$. Since
$$\sum_{1\leq i\leq p}|\sfs_i|=\sum_{1\leq j\leq p'}\mu_j+n\la_n=\sum_{1\leq i\leq n-1}i\nu_i+n\la_n=\sum_{1\leq i\leq n}\la_i=r,$$
we have $\bfs\in\ms S_r^{(n)}$. It is easy to see that $\pa_{n,r}(\bfs)=\bfQ$. Thus $\pa_{n,r}$ is surjective.
\end{proof}

\begin{Thm}\label{main theorem}
For $\bfs=\{\sfs_1,\ldots,\sfs_p\}\in\ms S_r^{(n)}$ with
$\sfs_i=(a_i\ttv^{-\mu_i+1},a_i\ttv^{-\mu_i+3},\ldots,a_i\ttv^{\mu_i-1}),$
we have $\sF(V_\bfs)\cong L(\bfQ_\bfs)$, where $\bfQ_\bfs=\pa_{n,r}(\bfs)$. In particular we have $\sF(V_\bfs)|_{\afUslC}\cong\bar L(\bfP)$, where
$$P_{i}(u)=\prod_{1\leq j\leq p\atop \mu_j=i}(1-a_ju).$$
for $1\leq
i\leq n-1$.
\end{Thm}
\begin{proof}
Let $W=\sF(\bar \sI_\mu)$.
By \ref{standard module} we have
$W\cong L(\bfQ_1)\ot\cdots\ot L(\bfQ_p)$, where $\bfQ_i=(Q_{i,1}(u),\cdots,Q_{i,n}(u))$ with $Q_{i,n}(u)=(1-a_i\ttv^{-n+1}u)^{\dt_{\mu_i,n}}$ and
$$P_{i,j}(u):=\frac{Q_{i,j} (u\ttv^{j-1})} {Q_{i,j+1} (u\ttv^{j+1})}=(1-a_iu)^{\dt_{j,\mu_i}}$$ for $1\leq i\leq p$ and $1\leq j\leq n-1$.
We will identify $W$ with $L(\bfQ_1)\ot\cdots\ot L(\bfQ_p)$. Let $w=w_1\ot\cdots\ot w_p\in W$, where $w_i$ is the pseudo-highest weight vector in $L(\bfQ_i)$. Then by \cite[6.3]{CP96} and \cite[4.1]{FM} we conclude that $w$ is the pseudo-highest weight vector in $W$ such that $\ttk_iw=\ttv^{\la_i}w$ and $\ms Q_i^\pm(u)w=Q_i^\pm(u)w$ for $1\leq i\leq n$, where $\la_i=\deg Q_i^+(u)$,
$$Q_n^\pm(u)=\prod_{1\leq i\leq p}Q_{i,n}^\pm(u)=\prod_{1\leq i\leq p}
(1-(a_iu)^{\pm 1}\ttv^{\pm(-n+1)})^{\dt_{\mu_i,n}}=\prod_{1\leq i\leq p\atop\mu_i=n}
(1-(a_iu)^{\pm 1}\ttv^{\pm(-n+1)})$$
and $$P_j^\pm(u):=\frac{Q_j^\pm(\ttv^{j-1}u)}{Q_{j+1}^\pm(\ttv^{j+1}u)}
=\prod_{1\leq i\leq p}P_{i,j}^\pm(u)=
\prod_{1\leq i\leq p}(1-(a_iu)^{\pm 1})^{\dt_{j,\mu_i}}
=\prod_{1\leq i\leq p\atop\mu_i=j}(1-(a_iu)^{\pm 1})$$
for $1\leq j\leq n-1$.
By definition we have  $\bfQ_\bfs=(Q_1^+(u),\cdots,Q_n^+(u))$. Since
$\la_j=\deg Q_j^+(u)=\la_n+\sum_{j\leq s\leq n-1}\deg P_s^+(u)=|\{1\leq i\leq p\mid\mu_i\geq j\}|$ for $1\leq j\leq n$, we have $\la=(\la_1,\cdots,\la_n)=\mu'$.

Let $L=\sF(V_\bfs)$.  Since $V_\bfs$ is a semisimple $\HrC$-module, by \ref{Lem1} and \ref{finite case} we have $[L:L(\la)]=[L:\OgnC^{\ot r}\ot_{\HrC} E_\mu]=[\OgnC^{\ot r}\ot_{\HrC}V_\bfs:\OgnC^{\ot r}\ot_{\HrC} E_\mu]=[V_\bfs:E_\mu]=1$.
Thus
\begin{equation}\label{dimension}
\dim L_\la=1.
\end{equation}
Since $V_\bfs$ is the irreducible subquotient of $\bar\sI_\mu$,  there is a surjective $\afUglC$-module homomorphism $f:M\ra L$, where $M$ is a certain submodule of $W$. Since  $1=\dim L_\la\leq\dim M_\la\leq\dim W_\la=1$, we conclude that $\dim M_\la=\dim W_\la=1$. Hence
$M_\la=W_\la=\spann\{w\}$ and $L_\la=\spann\{f(w)\}$. By \eqref{dimension} we have $f(w)\not=0$. Since $f$ is a $\afUglC$-module homomorphism, $f(w)$ is the pseudo-highest weight vector in $L$ such that $\ttk_if(w)=f(\ttk_iw)=\ttv^{\la_i}f(w)$ and $\sQ_i^\pm(u)f(w)=f(\sQ_i^\pm(u)w)=Q_i^\pm(u)f(w)$ for $1\leq i\leq n$. This implies that $L$ is the irreducible quotient module of $M(\bfQ_\bfs)$ and hence $L\cong L(\bfQ_\bfs)$.
\end{proof}

Combining  \ref{classification 1}, \ref{bijective map} with \ref{main theorem} yields the following classification theorem of irreducible $\afSrC$-modules, which was proved in \cite[4.6.8]{DDF} using a different approach.
\begin{Coro}\label{classification 2}
The set $\{L(\bfQ)\mid\bfQ\in\sQ(n)_r\}$ is a complete set of nonisomorphic finite dimensional irreducible $\afSrC$-modules.
\end{Coro}

Finally we will use \ref{main theorem} to generalize \cite[(6.5f)]{Gr80} to the affine case in \ref{prop of functor sG}.
Assume $N\geq n$. Let $e=\sum_{\la\in\Lanr}\bffkk_{\la}\in\sS_\vtg(N,r)_\mbc$. Then $e\sS_\vtg(N,r)_\mbc e\cong\afSrC$.
Consequently, the categories $e\sS_\vtg(N,r)_\mbc e\hmod$ and $\afSrC\hmod$ may be identified. With this identification, we
define a functor
\begin{equation}\label{functor sG}
\sG=\sG_{N,n,r}:\sS_\vtg(N,r)_\mbc\hmod\lra\afSrC\hmod,\qquad
V\longmapsto eV .
\end{equation}
 Then by definition we have
$\sG_{N,n,r}\circ\sF_{N,r}=\sF_{n,r}$.
 For $\bfQ=(Q_1(u),\cdots,Q_n(u))\in\sQ(n)_r$ let $\ti\bfQ=
(Q_1(u),\cdots,Q_n(u),1,\cdots,1)\in\sQ(N)_r$. Let $\ti\sQ(n)_r=\{\ti\bfQ\mid\bfQ\in\sQ(n)_r\}\han\sQ(N)_r$.
Clearly, by definition, we have
\begin{equation}\label{panr paNr}
\pa_{N,r}(\bfs)=\ti{\pa_{n,r}(\bfs)}.
\end{equation}
for $\bfs\in\mathscr S_r^{(n)}$.

\begin{Thm}\label{prop of functor sG}
Assume $N\geq n$.  Then $\sG(L(\ti\bfQ))\cong L(\bfQ)$ for $\bfQ\in\sQ(n)_r$. In particular we have
$\dim L(\ti\bfQ)_\al=\dim L(\bfQ)_\al$ for $\al\in\Lanr$.
Furthermore, for $\bfQ'\in\sQ(N)_r$,  $\sG(L(\bfQ'))\not=0$ if and only if  $\bfQ'\in\ti\sQ(n)_r$.
\end{Thm}
\begin{proof}
If $\bfQ\in\sQ(n)_r$ then by \ref{bijective map} we conclude that there exist $\bfs\in\ms S_r^{(n)}$ such that $\bfQ=\pa_{n,r}(\bfs)$.
By \ref{main theorem} and \eqref{panr paNr} we have $L(\ti\bfQ)\cong\sT_\vtg(N,r)\ot_{\afHrC}V_\bfs$. So by \cite[4.3.3]{DDF} and \ref{main theorem} we have $\sG (L(\ti\bfQ))\cong (e\sT_\vtg(N,r))\ot_{\afHrC}V_\bfs\cong \sT_\vtg(n,r)\ot_{\afHrC}V_\bfs\cong L(\bfQ)$.
By \cite[6.2(g)]{Gr80}, the set $\{\sG (L(\bfQ'))\not=0\mid\bfQ'\in\sQ(N)_r\}$ forms a complete set of non-isomorphic irreducible $\afSrC$-modules.
This together with \ref{classification 2} implies that
$\{\sG (L(\bfQ'))\not=0\mid\bfQ'\in\sQ(N)_r\}=\{
\sG(L(\ti\bfQ)) \mid\bfQ\in\sQ(n)_r\}$. Consequently,
$\sG(L(\bfQ'))\not=0$ if and only if  $\bfQ'\in\ti\sQ(n)_r$.
\end{proof}

\end{document}